\documentclass{amsart}

\usepackage{graphicx}
\usepackage{amsrefs}

\newtheorem{theorem}{Theorem}
\newtheorem{proposition}[theorem]{Proposition}
\newtheorem{lemma}[theorem]{Lemma}
\theoremstyle{definition}
\newtheorem{remark}[theorem]{Remark}

\numberwithin{theorem}{section}
\numberwithin{equation}{section}

\newcommand{\R}{\mathbb{R}}
\newcommand{\Z}{\mathbb{Z}}
\newcommand{\N}{\mathbb{N}}
\newcommand{\ep}{\varepsilon}
\newcommand{\trace}{\operatorname{trace}}
\newcommand{\abs}[1]{|#1|}

\begin{document}

\title{Convergence of the Abelian Sandpile}

\author{Wesley Pegden}
\address{Courant Institute at New York University, 251 Mercer Street, New York, NY 10012}
\email{pegden@math.nyu.edu}

\author{Charles K. Smart}
\address{Courant Institute at New York University, 251 Mercer Street, New York, NY 10012}
\email{csmart@math.nyu.edu}

\date{\today}
\thanks{The authors were partially supported by NSF grants DMS-1004696 and DMS-1004595.}

\begin{abstract}
The Abelian sandpile growth model is a diffusion process for configurations of chips placed on vertices of the integer lattice $\mathbb{Z}^d$, in which sites with at least $2d$ chips {\em topple}, distributing $1$ chip to each of their neighbors in the lattice, until no more topplings are possible.  From an initial configuration consisting of $n$ chips placed at a single vertex, the rescaled stable configuration seems to converge to a particular fractal pattern as $n\to \infty$.  However, little has been proved about the appearance of the stable configurations.   We use PDE techniques to prove that the rescaled stable configurations do indeed converge to a unique limit as $n \to \infty$.  We characterize the limit as the Laplacian of the solution to an elliptic obstacle problem.
\end{abstract}

\maketitle

%%%%%

\section{Introduction}
\label{introduction}

The Abelian sandpile growth model is a diffusion process for configurations of chips placed on vertices of the $d$-dimensional integer lattice $\Z^d$.  When a vertex in $\Z^d$ \emph{topples}, it loses $2d$ chips itself, adding one chip to each of its $2d$ neighbors in the lattice.    For any finite nonnegative initial distribution of chips, one can topple vertices with at least $2d$ chips until no more such topplings are possible, and, remarkably, neither the final configuration nor the number of topplings which occur at each vertex depend on the order in which topplings are performed --- the process is in this sense Abelian, as first noticed in  \cite{Dhar, Diaconis-Fulton}.

Introduced in \cite{Bak-Tang-Wiesenfeld}, the sandpile growth model has been the subject of extensive study over the past two decades, both on $\Z^d$ and on general graphs.  Nevertheless, many of the earliest and most fundamental questions regarding the sandpile remain unanswered (see e.g. \cite{Levine-Propp}).   Some of the most nagging concern the final configuration $s_n:\Z^d\to \{0,1,\dots,2d-1\}$ obtained from starting with $n$ chips placed at the origin of the integer lattice.  Terminal configurations for several values of $n$ and $d = 2$ are shown in Figure \ref{sandsequence}.  Rescaled by a factor of $n^{1/d}$, the piles seem to converge to a fractal-like limiting image, but essentially no progress has been made in proving precise things about this limit. (For example, there is no proof that there are regions in the limit which have constant value 3, no proof that its boundary does not converge to a circle, etc.)  In this paper, we show at least that the limit exists as an object to study --- the rescaled sandpile does converge.

\begin{figure}
\begin{center}
\includegraphics[width=.35\linewidth]{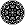}\hspace{1em}
\includegraphics[width=.35\linewidth]{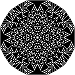}\\
\vspace{1em}
\includegraphics[width=.35\linewidth]{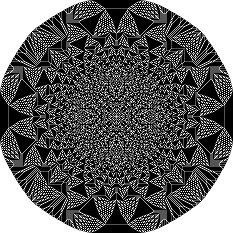}\hspace{1em}
\includegraphics[width=.35\linewidth]{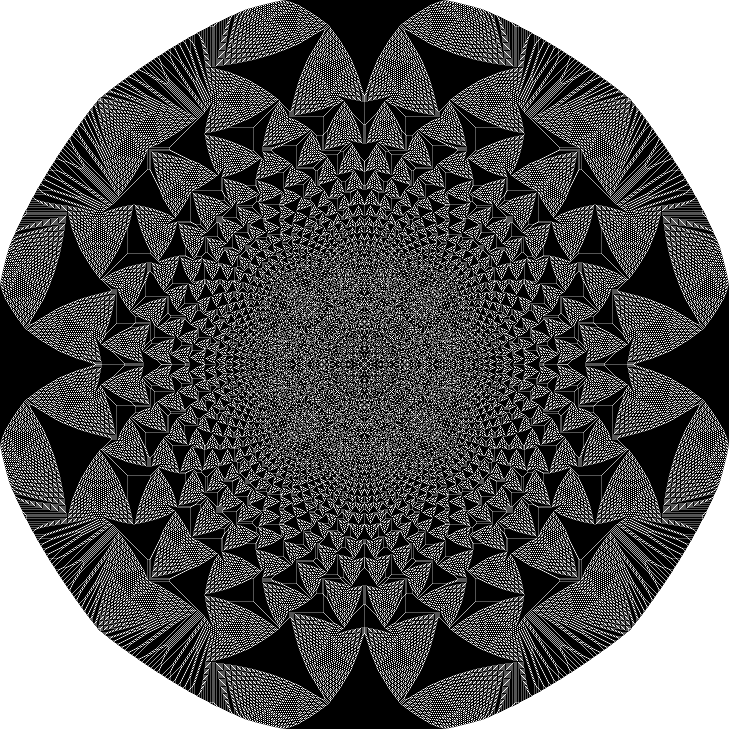}
\end{center}
\caption{Sandpiles started from $n = 10^3$, $n = 10^4$, $n = 10^5$, and $n = 10^6$ chips placed at the origin of $\Z^2$, rescaled by factors of $n^{1/2}$.  Sites with 0, 1, 2, and 3 chips are represented by white, light gray, dark gray, and black, respectively.\label{sandsequence}
} 
\end{figure}

\begin{theorem}
\label{main}
The rescaled sandpiles $\bar s_n(x) := s_n(n^{1/d}x)$ converge weakly-$*$ to a function $s \in L^\infty(\R^d)$ as $n \to \infty$. Moreover, the limit $s$ satisfies $\int_{\R^d} s \,dx = 1$, $0 \leq s \leq 2d - 1$, and  $s = 0$ in $\R^d \setminus B_R$ for some $R > 0$.
\end{theorem}

\noindent See Remark \ref{characterization} for the characterization the limiting $s \in L^\infty(\R^d)$ as the Laplacian of the solution of an elliptic obstacle problem.

We remark that weak-$*$ convergence seems to be the correct notion of convergence for the rescaled sandpiles. 
Recall that $L^\infty(\R^d)$ denotes the space of bounded measurable functions on $\R^d$ and that $C_0(\R^d)$ denotes the continuous functions on $\R^d$ with compact support. A sequence of functions $\bar s_n\in L^\infty(\R^d)$ converges weakly-$*$ to a function $s\in L^\infty(\R^d)$ in $L^\infty(\R^d)$ if
\begin{equation*}
\int_{\R^d} \bar s_n \varphi \,dx \to \int_{\R^d} s \varphi \,dx \quad \mbox{as } n \to \infty,
\end{equation*}
for all test functions $\varphi\in C_0(\R^d)$. If we look closely at the gray regions of the sandpiles (see Figure \ref{sandzoom}), we see that they consist of rapidly oscillating patterns, as first noticed in \cite{Ostojic}.  It seems clear that no kind of pointwise convergence of the sandpiles can hold, assuming this behavior persists as $n\to \infty$.  Instead, the sandpiles converge in the sense that the oscillating regions converge to their average value; there is some limiting image $s$ in which the ``colors'' of points are real numbers, approximated in the sequence $\bar s_n$ by patterns of integers.  The convergence of a sequence of bounded functions in local average value is precisely what weak-$*$ convergence in $L^\infty(\R^d)$ is designed to capture \cite{Evans}.

\begin{figure}
\begin{center}
\includegraphics[width=.8\linewidth]{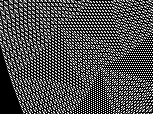}
\end{center}
\caption{Periodic structure in the ``gray'' regions of the $n = 16 \cdot 10^6$ sandpile on $\Z^2$.  Notice the patterned triangular shapes. \label{sandzoom}}
\end{figure}

Note that the properties of $s$ stated in the second part of the theorem are to demonstrate that the weak-$*$ convergence is not for trivial reasons such as a poor choice of the rescaling: the mass of the limit is 1 and is not being ``rescaled away''; similarly, the rescaled sandpiles are contained within a ball of a fixed size. The best known upper bound on the diameter of the this ball \cite{Levine-Peres} is that for every $\ep > 0$ there is a $C = C(\ep, d) > 0$ such that
\begin{equation}
\label{lp-bound}
\{ s_n > 0 \} \subseteq B_{((d-\ep) |B_1|)^{- 1/d}n^{1/d} + C}
\end{equation}
for all $n > 0$. Here $|B_1|$ denotes the volume of the unit ball.

\bigskip

Although our goal will be the characterization of final chip configurations, we will really be studying the number of topples which occur at each vertex as an initial configuration stabilizes. Given an initial configuration of chips and some finite sequence $\{p_i\}\subset \Z^d$ of points at which we topple piles of chips, the sequence is \emph{legal} if the sequence never topples any vertex with fewer than $2d$ chips, and \emph{stabilizing} if the final configuration has fewer than $2d$ chips at each vertex.  Given any legal sequence $p_1,\dots,p_s$ and any stabilizing sequence $q_1,\dots,q_t$ both beginning from some fixed intial configuation, we must have $p_1=q_{i_1}$ for some $i_1$, since there must be at least $2d$ chips at the point $p_1$ in the initial configuration, since $\{p_i\}$ is legal, and these chips must topple in any stabilizing sequence.  Since the permuted sequence $q_{i_1},q_1,\dots,q_{{i_1}-1},q_{{i_1}+1},\dots,q_t$ is also stabilizing, we can apply the same argument again from the configuration obtained after toppling just at $p_1=q_{i_1}$, to get that there must be a $q_{i_2}=p_2$.  Continuing in this manner, we obtain that the sequence $\{p_i\}$ is a permutation of some subsequence of $\{q_i\}$.

The above argument, a simplification of what appeared in \cite{Diaconis-Fulton}, implies that for any fixed initial configuration of chips, any two legal stabilizing sequences must be permutations of each other.  Coupled with the fact that any finite initial configuration of chips on $\Z^d$ will eventually stabilize after some legal sequence of topples, this implies that there is a well-defined \emph{odometer function} $v:\Z^d\to \N$ associated to any initial configuration, which counts the number of topples which will occur at points of the lattice in any legal stabilizing sequence of topples.  This also implies the Abelian property of the sandpile, as the final chip configuration $s$ is uniquely determined by the odometer function.  Note that it can be computed simply as
\begin{equation}
s(x)=\eta(x)+\sum_{y\sim x}(v(y)-v(x)),
\end{equation}
where $\eta$ is the initial configuration of chips, and the sum is taken over the $2d$ neighbors $y$ of $x$ in the lattice $\Z^d$.  Thus we have 
\begin{equation*}
s=\eta+\Delta^1v,
\end{equation*}
where $\Delta^1$ is the $(2d+1)$-point discrete Laplacian on $\Z^d$.

The starting point of our proof is the {\em least action principle} formulated in \cite{Fey-Levine-Peres}.  Suppose that $u$ is the odometer function for some initial configuration $\eta$ of chips, and $v:\Z^d\to \N$ is any function that satisfies 
\begin{equation*}
\eta + \Delta^1 v \leq 2d - 1.
\end{equation*}
The least action principle states that we must have $u\leq v$.  Note that this follows immediately from the fact that legal sequences are permutations of subsequences of stabilizing sequences, since $u$ corresponds to a stabilizing \emph{and} legal sequence, while $v$ corresponds to a stablizing, not necessarily legal sequence.

The important consequence of the least action principle for us is that the odometer function associated with a configuration is the pointwise minimum of all functions $v:\Z^d\to \N$ satisfying $\eta+\Delta^1v\leq 2d-1$, where $\eta:\Z^d\to \N$ is the initial configuration of chips.  In our case, if $v_n$ is the odometer function resulting from an initial configuration of $n$ chips at the origin, we have
\begin{equation}
v_n=\min\{v : \Z^d\to \N\mid n\delta_0+\Delta^1v\leq 2d-1\},
\label{l.least}
\end{equation}
where $\delta_0$ is the characteristic function of the set $\{0\}\subset \Z^d$.  The final configuration of chips is then
\begin{equation*}
s_n:=n\delta_0+\Delta^1v_n.
\end{equation*}

This description of $v_n$ and $s_n$, together with standard estimates for the $(2d+1)$-point Laplacian $\Delta^1$, will easily give convergence of the rescaled $\bar s_n$ and the rescaled odometer function $\bar v_n(x) := n^{-2/d} v_n(n^{1/d} x)$ along \emph{subsequences} $n_k \to \infty$.

To obtain convergence, we show that the limiting $s$ and $v$ are independent of the choice of subsequence.  Assuming there are two distinct limits $v$ and $v'$, we use the regularity theory of the Laplacian to select a point $x \in \R^d \setminus \{ 0 \}$ where $v$ and $v'$ are sufficiently smooth and $v - v'$ is strictly concave. We then select sufficiently close finite approximations $v \approx \bar v_n$ and $v' \approx \bar v_{n'}$ and use $\bar v_{n'}$ to ``lower'' $\bar v_n$, contradicting the least action principle for $v_n$. The difficultly lies in the fact that $\bar v_n$ and $\bar v_{n'}$ are defined on possibly incompatible lattices.  We overcome this obstacle using an approximation argument (see Lemma \ref{approx}) that allows us to change the scale of $\bar v_{n'}$.

\bigskip

\noindent {\em Acknowledgements.} The authors would like to thank Robert V. Kohn for several useful discussions and Lionel Levine for drawing our attention to the problem and for making comments on an early draft.

%%%%%

\section{Preliminaries}
\label{preliminaries}

%%%%%

\subsection{Notation}

We write $h := n^{- 1/d}$. We define the open ball $B_r(x) := \{ y \in \R^d \mid |y - x| < r \}$ for $r > 0$ and $x \in \R^d$ and define $B_r := B_r(0)$.  We write $\partial \Omega$ for the boundary of an open set $\Omega \subseteq \R^d$.  A function $u \in C(\R^d)$ is twice differentiable at a point $x \in \R^d$ if there is a vector $Du(x) \in \R^d$ and a symmetric $d \times d$ matrix $D^2 u(x) \in S_d$ such that
\begin{equation*}
u(y) = u(x) + Du(x) \cdot (y - x) + \frac{1}{2} (y - x)^t D^2 u(x) (y - x) + o(|y - x|^2).
\end{equation*}
We write $y \sim_h x$ when $y - x \in h \Z^d := \{ h x \mid x \in \Z^d \}$ and $|y - x| = h$. Note that $y \sim_h x$ if and only if $y = x \pm h e_i$ for some coordinate vector $e_i$. Given a subset of the lattice $E \subseteq h \Z^d$, we write
\begin{equation*}
\partial^h E := \{ y \in h \Z^d \setminus E \mid y\sim_h x \mbox{ for some }x\in E\}
\end{equation*}
for its lattice boundary.

%%%%%

\subsection{Interpolation convention}

Throughout this article, we consider sequences of functions $u_n : h \Z^d \to \R$. We implicitly extend all such functions to $\R^d$ via nearest-neighbor interpolation. That is, we set
\begin{equation*}
u_n(x) := u_n(h \lfloor h^{-1} x \rceil)
\end{equation*}
to be the value of $u_n$ at the lattice point $h \lfloor h^{-1} x \rceil \in h \Z^d$ closest to $x \in \R^d$. Here $\lfloor \cdot \rceil$ denotes coordinate-wise rounding to the nearest integer (rounding down for ties).  This convention allows us to make sense of statements like ``$u_n \to u \in C(\R^d)$ locally uniformly as $n \to \infty$'' when the functions $u_n$ are only defined on the lattice $h \Z^d$.

%%%%%

\subsection{The Laplacian}

The Laplacian
\begin{equation*}
\Delta u := \trace(D^2 u) = \sum_{i = 1}^d \frac{\partial^2}{\partial x_i^2} u,
\end{equation*}
and the $(2d+1)$-point discrete Laplacian
\begin{equation*}
\Delta^h u(x) := \frac{1}{h^2} \sum_{y\sim_h x} (u(y) - u(x))
\end{equation*}
play a central role in our analysis of the sandpile.  Recall that if $\varphi \in C^\infty(\R^d)$, then $\Delta^h \varphi \to \Delta \varphi$ locally uniformly in $\R^d$ as $h \to 0$.

From \cite{Lawler-Limic}, we know that the discrete Laplacian $\Delta^h$ has a ``fundamental solution'' $\Phi_n : h \Z^d \to \R$ that satisfies
\begin{equation*}
\Delta^h \Phi_n(x) = \begin{cases} - n & \mbox{if } x = 0, \\ 0 & \mbox{if } x \neq 0, \end{cases}
\end{equation*}
for all $x \in h \Z^d$. The function $\Phi_n(x)$ can be realized as a normalization of the expected number of visits of a random walk to the point $x$ on the lattice $h\Z^d$, adjusting for recurrences in the case $d = 2$. As $n \to \infty$, $\Phi_n \to \Phi \in C^\infty(\R^d \setminus \{ 0 \})$ locally uniformly in $\R^d \setminus \{ 0 \}$, where
\begin{equation*}
\Phi(x) := \begin{cases}
(d (d-2) |B_1|)^{-1} |x|^{2 - d} & \mbox{if } d \geq 3, \\
- (2 \pi)^{-1} \log |x| & \mbox{if } d = 2,
\end{cases}
\end{equation*}
is the ``fundamental solution'' of the Laplacian. We use the convergence of $\Phi_n \to \Phi$ to resolve the formation of a singularity in $\bar v_n$ as $n \to \infty$.

The discrete Laplacian is {\em monotone}: $\Delta^h u(x)$ is decreasing in $u(x)$ and increasing in $u(y)$ for any $y \sim_h x$. An obvious consequence of monotonicity is the following estimate for the discrete Laplacian of the point-wise minimum of two functions:

\begin{proposition}
\label{hlap-inf}
If $u, v : h \Z^d \to \R$, $w := \min \{ u, v \}$, and $w(x) = u(x)$, then $\Delta^h w(x) \leq \Delta^h u(x)$.
\end{proposition}

\noindent A less obvious consequence of monotonicity is the fact that $\Delta^h$ has a maximum principle \cite{Kuo-Trudinger}:

\begin{proposition}
\label{hlap-compare}
For $u, v : h \Z^d \to \R$, $E \subseteq h \Z^d$ finite, and $\Delta^h u \geq \Delta^h v$ in $E$, we have $\max_{\partial^h E} (u - v)\geq \max_E (u - v)$.  That is, the maximum difference must occur on the lattice boundary.
\end{proposition}

Finally, we recall a consequence of the standard a priori estimates for the finite difference Laplacian and the Arzela-Ascoli theorem \cite{Kuo-Trudinger}:

\begin{proposition}
\label{hlap-precompact}
If $u_n : h \Z^d \to \R$ is a sequence of functions that satisfies
\begin{equation*}
\max_{B_1}\abs{u_n}\leq C \quad \mbox{and}\quad \max_{B_1} \abs{\Delta^h u_n } \leq C,
\end{equation*}
then for every sequence $n_k \to \infty$ there is subsequence $n_{k_j} \to \infty$ and a function $u \in C(B_1)$ such that $u_{n_{k_j}} \to u$ locally uniformly in $B_1$ as $j \to \infty$.
\end{proposition}

After we make some additional preparations, the above two propositions will allow us to show that the $\bar v_n$ converge along subsequences. It remains to recall some standard facts about the continuum Laplacian which we use to study the regularity of the possible limits of the $\bar v_n$.

If $\Omega \subseteq \R^d$ is open and $s \in L^\infty(\Omega)$, then we say that $u \in C(\Omega)$ is a {\em weak solution} (see \cite{Gilbarg-Trudinger}) of the continuum Laplace equation
\begin{equation*}
\Delta u = s \mbox{ in } \Omega,
\end{equation*}
if and only if
\begin{equation}
\int_{\Omega} u \Delta \varphi \,dx = \int_{\Omega} s \varphi \,dx
\label{l.lapparts}
\end{equation}
for all test functions $\varphi \in C^\infty_0(\Omega)$. Note that if $u$ and $s$ happen to be smooth, then we can we can intergrate by parts to show that this is equivalent to $\Delta u = s$ holding pointwise.

The following is a consequence of a much stronger result from the theory of singular operators  \cite[Chapter III Theorem 4]{Christ}:

\begin{proposition}
\label{bmo}
If $s\in L^\infty(B_1)$ and $u \in C(B_1)$ is a weak solution of $\Delta u=s$ in $B_1$, then $u$ is twice differentiable and $\Delta u = s$ almost everywhere in $B_1$.
\end{proposition}

Having bounded Laplacian also implies that a function must be strictly concave on a set of positive measure in any neighborhood of a strict local maximum \cite[Theorem 3.2]{Caffarelli-Cabre}:

\begin{proposition}
\label{abp}
If  $s \in L^\infty(B_1)$ and $u \in C(B_1)$ is a weak solution of $\Delta u = s$ in $B_1$, and $\sup_{B_1} u > \sup_{\partial B_1} u$, then the set
\begin{equation*}
\{ x \in B_1 \mid u(x) > \sup_{\partial B_1} u, \ u \mbox{ is twice differentiable at } x, \mbox{ and } D^2 u(x) < 0 \},
\end{equation*}
has positive measure.
\end{proposition}

%%%%%

\section{Convergence along subsequences}

Note from Section \ref{introduction}, and in particular \eqref{lp-bound}, that the sandpile $s_n$ satisfies
\begin{equation}
\label{sninequalities}
\int_{\R^d} s_n \,dx = n, \quad 0 \leq s_n \leq 2d - 1 \quad \mbox{and} \quad \{ s_n > 0 \} \subseteq B_{h^{-1} R},
\end{equation}
for some $R > 0$ independent of $n$.  These facts, together with the least action principle and the results in Section \ref{preliminaries}, are enough to prove convergence of the rescaled sandpile along subsequences.

Recall that $\bar s_n(x) := s(h^{-1} x)$ and $\bar v_n(x) := h^2 v_n(h^{-1} x)$, and define
\begin{equation*}
\bar w_n := \bar v_n - \Phi_n.
\end{equation*}
Since $\Delta^h$ is linear and $\Delta^h \Phi_n = - n \delta_0$, we have
\begin{equation*}
\Delta^h \bar w_n = n \delta_0 + \Delta^h\bar v_n =\bar s_n \quad \mbox{in } \R^d.
\end{equation*}

\begin{lemma}
\label{subsequence}
For every sequence $n_k \to \infty$ there is a subsequence $n_{k_j} \to \infty$ and functions $w \in C(\R^d)$ and $s \in L^\infty(\R^d)$ such that $w$ is a weak solution of $\Delta w = s$ in $\R^d$, $\bar w_{n_{k_j}}$ converges locally uniformly to $w$ in $C(\R^d)$, and $\bar s_{n_{k_j}}$ converges weakly-$*$ to $s$ in $L^\infty(\R^d)$.
\end{lemma}

Note that this lemma holds for any sequence $w_n : h \Z^d \to \R$ such that $w_n$ and $\Delta^h w_n$ are locally uniformly bounded.  Uniqueness of the limit is where our analysis is particular to the sandpile.

\begin{proof}
Since $\{ \bar v_n > 0 \} \subseteq B_R$ and $\bar v_n \geq 0$, we therefore have
\begin{equation*}
\bar w_n \geq - \Phi_n \quad \mbox{in } h \Z^d \quad \mbox{and} \quad \bar w_n = - \Phi_n \quad \mbox{on } h \Z^d \setminus B_R.
\end{equation*}
If we set $E := h \Z^d \cap B_R$, then the test functions
\begin{equation*}
\varphi(x) := |x|^2 - (R + h)^2 + \inf_{\partial^h E} - \Phi_n \quad \mbox{and} \quad \psi(x) := \sup_{\partial^h E} - \Phi_n,
\end{equation*}
satisfy $\varphi \leq \bar w_n \leq \psi$ on $\partial^h E$ and $2d = \Delta^h \varphi \geq \Delta^h \bar w_n \geq \Delta^h \psi = 0$ in $E$.  It follows from the maximum principle (Proposition \ref{hlap-compare}) that
\begin{equation*}
|x|^2 - (R+h)^2 + \inf_{\partial^h E} - \Phi_n \leq \bar w_n(x) \leq \sup_{\partial^h E} - \Phi_n \quad \mbox{for all } x \in B_R.
\end{equation*}
For any $R' > R$, we know that $\Phi_n \to \Phi$ uniformly in $B_{R'} \setminus B_R$ as $n \to \infty$. Thus the sequence $\bar w_n$ satisfies
\begin{equation*}
\| \bar w_n \|_{L^\infty(B_{R'})} \leq C
\end{equation*}
for some $C = C(R') > 0$ and all $n > 0$.  Since we also have $|\Delta^h \bar w_n| \leq 2d - 1$, Proposition \ref{hlap-precompact} implies the existence of a subsequence $n_{k_j}$ and a function $w \in C(\R^d)$ such that $\bar w_{n_j} \to w$ locally uniformly as $j \to \infty$. Since any uniformly bounded sequence of functions converges weakly-$*$ along subsequences \cite[Page 7]{Evans}, we may select a further subsequence and a function $s \in L^\infty(\R^d)$ such that $\bar s_{n_j}$ converges weakly-$*$ to $s$ as $j \to \infty$.

For any $\varphi \in C^\infty_0(\R^d)$, we have
\begin{equation}
\int_{\R^d} \varphi s_{n_j} \,dx= \int_{\R^d} \varphi \Delta^{h_j} \bar w_{n_j}\, dx= \int_{\R^d} (\Delta^{h_j} \varphi )\bar w_{n_j} \,dx,
\label{l.dparts}
\end{equation}
where the second equality comes from the discrete integration by parts formula, which an be carried out in this case by writing $\Delta^{h_j} \bar w_n$ as a finite sum that commutes with the integral.

Since $\bar w_{n_j} \to w$ locally uniformly in $\R^d$, it converges uniformly on any closed ball, and therefore uniformly on any bounded neighborhood of the support of $\varphi$.  Since we also have $\Delta^h \varphi \to \Delta \varphi$ uniformly in $\R^d$, we see that the right-hand side of (\ref{l.dparts}) converges to $\int_{\R^d} w \Delta \varphi \,dx$ as $j \to \infty$. The left-hand side converges to $\int_{\R^d} s \varphi \,dx$ by the definition of weak-$*$ convergence. Thus $w$ is a weak solution of $\Delta w = s$ in $\R^d$ by definition.
\end{proof}

An immediate consequence of Lemma \ref{subsequence} is that $\bar v_{n_j} \to v := w + \Phi$ locally uniformly in $\R^d \setminus \{ 0 \}$ as $j \to \infty$. However, what we obtained is strictly stronger, since it allows us to resolve the structure of the singularity of $v$.

%%%%%

\section{Convergence}

As discussed in Section \ref{introduction}, we would like to compare two limits $\bar v_{n_k} \to v$ and $\bar v_{n'_k} \to v'$ along different sequences $n_k \to \infty$ and $n'_k \to \infty$ by comparing their finite approximations, in spite of the fact that their approximations $\bar v_{n_k}$ and $\bar v_{n'_k}$ may not be defined on the same scale.  To deal with this, we could try construct an {\em asymptotic expansion} of the convergence $\bar v_{n_k} \to v$ in a neighborhood of each point $x \in \R^d \setminus \{ 0 \}$ where $v$ is twice differentiable. That is, we could try to find a radius $r > 0$ and a function $u : \Z^d \to \Z$ such that $\Delta^1 u \leq 2d - 1$ and
\begin{equation}
\label{expand}
\bar v_{n_k}(y) = v(x) + Dv(x) \cdot (y - x) + h_k^2 u(h_k^{-1}(y - x)) + o(|y - x|^2) + o(1),
\end{equation}
for all $y \in B_r(x)$ and $k > 0$. (Note that $o(1)$ is with respect to $k \to \infty$ while $o(|y - x|^2)$ is with respect to the distance $|y - x| \to 0$.) The existence of such an expansion would make it easy to approximate $v$ on different scales in $B_r(x)$. Indeed, while $v$ only has approximations $\bar v_{n_k}$ defined on the scales $h_k$, the right-hand side of \eqref{expand} is valid for any scale $h > 0$.

Unfortunately, we do not know how to prove that an asymptotic expansion of the form \eqref{expand} exists in general.  Instead we construct a one-sided expansion, which is sufficient for our purposes. Note that this lemma is applied below to $v_n$ that have been translated and shifted and in particular may take negative values.

\begin{lemma}
\label{approx}
Suppose $v_n : \Z^d \to \Z$ is any sequence of functions such that $\Delta^1 v_n \leq K \in \Z$. Suppose $v \in C(B_r(x_0))$ for some $r > 0$ and the rescalings $\bar v_n(x) := h^{2} v_n(h^{-1} x)$ converge uniformly $\bar v_{n_k} \to v$ in $B_r(x_0)$ along some sequence $n_k \to \infty$. If $v$ is twice differentiable at $x_0$, then for every $\ep > 0$ there is a function $u : \Z^d \to \Z$ such that
\begin{equation*}
\Delta^1 u(x) \leq K \quad \mbox{and} \quad u(x) \geq \frac{1}{2} x^t (D^2 v(x_0) - \ep I) x \quad \mbox{for all } x \in \Z^d.
\end{equation*}
\end{lemma}

\begin{proof}
Replacing $v_n$ if necessary by
\begin{equation*}
v_n'(x) := v_n(x + \lfloor h^{-1} x_0 \rceil) - v_n(\lfloor h^{-1} x_0 \rceil)- \lfloor h^{-1} D v(x_0) \rceil \cdot x,
\end{equation*}
we may assume that $x_0 = 0$, $v(0) = 0$, $Dv(0) = 0$, and $v_n(0) = 0$.

Since $v$ is twice differentiable at $0$, we can make $r > 0$ smaller and select a large $n = n_k$ such that
\begin{equation*}
\sup_{B_r} |\bar v_n - \varphi| \leq \ep r^2,
\end{equation*}
where $\varphi(x) := \frac{1}{2} x^t D^2 v(0) x$. Noting that $h^2\varphi(h^{-1}x)=\varphi(x)$, we can undo the scaling of $\bar v_n$, obtaining
\begin{equation*}
\sup_{B_{h^{-1} r}} |v_n - \varphi| \leq \ep h^{-2} r^2.
\end{equation*}
Thus, if we define $\psi(x) := \frac{1}{2}x^t(D^2 v(0) - 32\ep I) x$, we have $\varphi-\psi=16\ep \abs x^2$ and therefore
\begin{equation}
\label{approx-bounds}
\begin{cases}
v_n \leq \psi +  2\ep h^{-2} r^2 & \mbox{in } B_{h^{-1} r/4} \\
v_n \geq \psi +  3\ep h^{-2} r^2 & \mbox{in } B_{h^{-1} r} \setminus B_{h^{-1} r/2}.
\end{cases}
\end{equation}

We define $u$ as an overlapping pointwise minimum of translated and tilted copies of $v_n$,
\begin{equation*}
u(x)  := \min \{  v_{n,y}(x) \mid y \in \Z^d \cap B_{h^{-1} r}(x) \},
\end{equation*}
where
\begin{equation*}
v_{n,y}(x) := v_n(x - y) + \lfloor D \psi(y) \rceil \cdot (x - y) + \lfloor \psi(y) \rceil.
\end{equation*}
The inequalities \eqref{approx-bounds} guarantee that the overlapping works out correctly.
Indeed, for $x \in B_{h^{-1} r / 4}(y)$, we compute
\begin{align*}
v_{n,y}(x) & = v_n(x - y) + \lfloor D \psi(y) \rceil \cdot (x - y) + \lfloor \psi(y) \rceil \\
& \leq \psi(x - y) + \lfloor D \psi(y) \rceil \cdot (x - y) + \lfloor \psi(y) \rceil + 2 \ep h^{-2} r^2 \\
& \leq  \psi(x - y) + D \psi(y) \cdot (x - y) + \psi(y) + 2 \ep h^{-2} r^2 + d^{1/2} h^{-1} r + 1 \\
& = \psi(x) + 2 \ep h^{-2} r^2 + d^{1/2} h^{-1} r + 1.
\end{align*}
A similar computation shows
\begin{equation*}
v_{n,y} \geq \psi + 3 \ep h^{-2} r^2 - d^{1/2} h^{-1} r - 1 \quad \mbox{in } B_{h^{-1} r}(y) \setminus B_{h^{-1} r/2}(y).
\end{equation*}
By making $n$ larger, we may assume that $h^{-1} r / 4 > 1$ and
\begin{equation*}
2\ep h^{-2} r^2 + d^{1/2} h^{-1} r + 1 < 3\ep  h^{-2} r^2 - d^{1/2} h^{-1} r - 1.
\end{equation*}
It follows that
\begin{equation*}
u(x) = \min \{ v_{n,y}(x) \mid y \in \Z^d \cap B_{h^{-1} r/2}(x) \},
\end{equation*}
and therefore
\begin{equation*}
u = \min \{ v_{n,y} \mid y \in \Z^d \cap B_{3h^{-1} r/4}(x) \} \quad \mbox{in } B_{h^{-1} r / 4}(x),
\end{equation*}
for all $x \in \Z^d$. Since $\Delta^1 v_{n,y} \leq K$ and $h^{-1} r/4 > 1$, Proposition \ref{hlap-inf} implies that $\Delta^1 u(x) \leq K$. Moreover, it is easy to check that $u - \psi$ is bounded from below.
\end{proof}

We are now in a position to prove Theorem \ref{main}. In fact, we prove something slightly stronger.

\begin{theorem}
\label{converge}
There are functions $w \in C(\R^d)$ and $s \in L^\infty(\R^d)$ such that $\bar s_n$ converges weakly-$*$ to $s$ in $L^\infty(\R^d)$, $\bar w_n$ converges locally uniformly to $w$ in $C(\R^d)$, and $w$ is a weak solution of $\Delta w = s$ in $\R^d$. Moreover, the function $s$ satisfies $\int_{\R^d} s \,dx = 1$, $0 \leq s \leq 2d - 1$, and $s = 0$ in $\R^d \setminus B_R$ for some $R > 0$.
\end{theorem}

\begin{proof}
Since Lemma \ref{subsequence} gives existence of limits along subsequences, we need only prove uniqueness of the limiting $s$ and $w$. By Proposition \ref{bmo}, $\Delta w = s$ almost everywhere, so it is enough to show that $w$ is unique. Suppose $w, w' \in C(\R^d)$ are distinct and that $\bar w_{n_k} \to w$ and $\bar w_{n'_k} \to w'$ along subsequences locally uniformly in $\R^d$ as $k \to \infty$. Since $w = w' = -\Phi$ outside of $B_R$ for some $R > 0$, we may assume without loss of generality that
\begin{equation*}
\sup_{B_R}  (w - w') > 0 = \sup_{\partial B_R} (w - w').
\end{equation*}
According to Proposition \ref{abp}, we may select a point $a \in \R^d$ such that $a \neq 0$, $w(a) > w'(a)$, both $w$ and $w'$ are twice differentiable at $a$, and $D^2 w(a) \leq D^2 w'(a) - 2 \ep I$ for some $\ep > 0$.

Define $v := w + \Phi$ and $v' := w' + \Phi$. Since $a \neq 0$, $\bar v_{n'_k} \to v'$ uniformly in a neighborhood of $a$ as $k \to \infty$ and $v'$ is twice differentiable at $a$. Using Lemma \ref{approx}, we may select a $u : \Z^d \to \Z$ such that
\begin{equation*}
\Delta^1 u(x) \leq 2d - 1 \quad \mbox{and} \quad u(x) \geq \frac{1}{2} x^t (D^2 v'(a) - \ep I) x \quad \mbox{for all } x \in \Z^d.
\end{equation*}
Since  $D^2 v'(a) - \ep I\geq D^2 v(a)+\ep I$ it follows that
\begin{equation*}
u(x) \geq \frac{1}{2} x^t (D^2 v(a) + \ep I) x \quad \mbox{for all } x \in Z^d.
\end{equation*}

We use a shifted version of $u$ to ``lower'' $v_{n_k}$. For the $a$ chosen above, we define
\begin{equation*}
h_k := n_k^{-1/d}, \quad a_k := \lfloor h_k^{-1} a \rceil,
\end{equation*}
and\begin{equation*}
u_{n_k}(x) := \begin{cases}
\begin{aligned}u(x - a_k) + \lfloor h_k^{-1} Dv(a) \rceil \cdot (x - a_k) \qquad \\ - u(0) + v_{n_k}(a_k) - 1 \end{aligned} & \mbox{if } x \in B_{h_k^{-1} r}(a_k), \\
1 + \max_{\Z^d} v_{n_k} & \mbox{otherwise.}
\end{cases}
\end{equation*}
We claim that for small $r > 0$ and large $n_k$, the function $\tilde v := \min \{ v_{n_k}, u_{n_k} \}$ contradicts the least action principle for $v_{n_k}$. Since
\begin{equation*}
\tilde v(a_k) = u_{n_k} (a_k) = v_{n_k}(a_k) - 1,
\end{equation*}
it is enough to show that $n \delta_0 + \Delta^1 \tilde v \leq 2d - 1$ and $\tilde v \geq 0$.

To show that $n \delta_0 + \Delta^1 \tilde v \leq 2d - 1$, it is enough, by Proposition \ref{hlap-inf}, to show that $v_{n_k} \leq u_{n_k}$ in $B_{h_k^{-1} r}(a_k) \setminus B_{h_k^{-1} r / 2}(a_k)$ and $0 \notin B_{h_k^{-1} r / 2}(a_k)$ for small $r > 0$ and large $n_k$. Since $v$ is twice differentiable at $a$, we have, for small $r > 0$ and large $n_k$,
\begin{equation*}
\bar v_{n_k}(x) \leq \bar v_{n_k}(a) + Dv(a) \cdot (x - a) + \frac{1}{2} (x-a)^t D^2 v(a) (x-a) + \frac{\ep}{8} r^2,
\end{equation*}
for all $x \in B_r(a)$. Undoing the scaling, we discover that
\begin{equation*}
v_{n_k}(x) \leq v_{n_k}(a_k) + h_k^{-1} Dv(a) \cdot (x - a_k) + \frac{1}{2} (x-a_k)^t D^2 v(a) (x-a_k) + \frac{\ep}{8} h_k^{-2} r^2,
\end{equation*}
for all $x \in B_{h_k^{-1} r}(a_k)$. Since
\begin{equation*}
u(x) \geq \frac{1}{2} x^t D^2 v(a) x + \frac{\ep}{4} h_k^{-2} r^2,
\end{equation*}
for all $x \in B_{h_k^{-1} r} \setminus B_{h_k^{-1} r/2}$, we see that
\begin{equation*}
u_{n_k}(x) \geq v_{n_k}(x) + \frac{\ep}{8} h_k^{-2} r^2 - d^{1/2} h_k^{-1} r - u(0) - 1,
\end{equation*}
for all $y \in B_{h_k^{-1} r}(a_k) \setminus B_{h_k^{-1} r/2}(a_k)$. Since the error term on the right-hand side is positive for large $n_k$, we see that
\begin{equation*}
u_{n_k} > v_{n_k} \quad \mbox{in } B_{h_k^{-1} r}(a_k) \setminus B_{h_k^{-1} r/2}(a_k),
\end{equation*}
for small $r > 0$ and large $n_k$.

To show that $\tilde v \geq 0$, it is enough to show that $u_{n_k} \geq 0$ in $B_{h_k^{-1} r}(a_k)$ for small $r > 0$ and large $n_k$. Note that since $v$ is twice differentiable at $a$, $\bar v_{n_k} \to v$ in $B_r(a)$, and $w(a) > w'(a)$, we have
\begin{equation*}
\bar v_{n_k}(a) + Dv(a) \cdot (x - a) + \frac{1}{2} (x - a)^t D^2 v(a) (x - a)  \geq \frac{v(a)}{2} > \frac{v'(a)}{2} \geq 0,
\end{equation*}
for small $r > 0$, large $n_k$, and $x \in B_r(a)$. Thus, for $x \in B_{h_k^{-1} r}(a_k)$, we may compute
\begin{align*}
u_{n_k}(x) & = u(x - a_k) + \lfloor h_k^{-1} Dv(a) \rceil \cdot (x - a_k) - u(0) + v_{n_k}(a_k) - 1 \\
& \geq \frac{1}{2}(x - a_k)^t D^2 v(a) (x - a_k) + h_k^{-1} Dv(a) \cdot (x - a_k) + v_{n_k}(a_k) \\ & \qquad \qquad - u(0) - 1 - d^{1/2} |x - a_k| \\
& \geq h_k^{-2} \frac{v(a)}{2} - u(0) - 1 - h_k^{-1} d^{1/2} \\
& \geq 0,
\end{align*}
for small $r > 0$ and large $n_k$.

The remaining assertions about $s$ follow immediately from \eqref{sninequalities} and our choice of scaling.
\end{proof}

\begin{remark}
\label{characterization}
For those readers familiar with viscosity solutions of fully nonlinear elliptic equations (see \cite{Caffarelli-Cabre}), we point out that the above proof shows that the limit $\bar w_n \to w$ is the unique solution of the obstacle problem
\begin{multline}
\label{obstacle}
w := \inf \{ w' \in C(\R^d) \mid w' \geq - \Phi \mbox{ and } \Delta w' \leq 2d - 1 \mbox{ in } \R^d, \\ \mbox{ and } G(D^2 w' + D^2 \Phi) \leq 0 \mbox{ in } \R^d \setminus \{ 0 \} \},
\end{multline}
where
\begin{multline*}
G(A) := \inf \{ s \in \R \mid \mbox{there is a } u : \Z^d \to \Z \mbox{ such that for all } y \in \Z^d \\ \Delta^1 u(y) \leq 2d - 1 \mbox{ and } u(y) \geq \frac{1}{2} y^t (A - s I) y \},
\end{multline*}
and we interpret the differential inequalities in the sense of viscosity. In fact, our proof that the limit $w$ is unique is exactly a proof that $w$ is the unique solution of \eqref{obstacle} with the viscosity solution terminology stripped out.

Without too much difficulty, one can show that the operator $G$ is Lipschitz and elliptic. However, we have no reason to believe that $G$ is {\em uniformly} elliptic, and thus the existing theory of fully nonlinear elliptic obstacle problems does not apply directly. Whether this theory can be re-worked for $G$ and whether the characterization \eqref{obstacle} is useful remains to be seen.
\end{remark}

%%%%%

\end{document}